\documentclass[11pt,table]{amsart}
\usepackage{amsmath}
\usepackage[margin=1.0in]{geometry}
\usepackage{amscd,amsmath,amsxtra,amsthm,amssymb,stmaryrd,xr,mathrsfs,mathtools,enumerate,commath, comment, enumitem, graphicx}
\usepackage{amsfonts, amssymb, amsthm, amsmath, braket, hyperref, mathtools, comment, mathrsfs, enumitem, cleveref}
\usepackage{stmaryrd}
\usepackage{orcidlink}
\usepackage{multirow}
\usepackage{xcolor}
\usepackage{commath}
\usepackage{comment}
\usepackage{tikz-cd}
\usepackage{tkz-graph}
\usepackage{colonequals}
\usepackage{hyperref}
\usepackage{graphicx}
\usepackage{bigstrut}
\usepackage{longtable, multirow, array}
\usepackage{enumitem}

\usepackage{longtable} 
\usepackage{pdflscape} 
\usepackage{booktabs}
\usepackage{hyperref}
\definecolor{vegasgold}{rgb}{0.77, 0.7, 0.35}
\definecolor{darkgoldenrod}{rgb}{0.72, 0.53, 0.04}
\definecolor{gold(metallic)}{rgb}{0.83, 0.69, 0.22}
\hypersetup{
	colorlinks=true,
	linkcolor=darkgoldenrod,
	filecolor=brown,      
	urlcolor=gold(metallic),
	citecolor=darkgoldenrod,
}

\usepackage[all,cmtip]{xy}
\DeclareFontFamily{U}{wncy}{}
\DeclareFontShape{U}{wncy}{m}{n}{<->wncyr10}{}
\DeclareSymbolFont{mcy}{U}{wncy}{m}{n}
\DeclareMathSymbol{\Sh}{\mathord}{mcy}{"58}
\usepackage[T2A,T1]{fontenc}
\usepackage[OT2,T1]{fontenc}

\usepackage{tikz}
\usetikzlibrary{shapes.geometric}
\usetikzlibrary{decorations.markings}
\tikzset{every loop/.style={min distance=10mm,looseness=10}}
\tikzstyle{vertex}=[auto=left,circle,minimum size=1pt,inner sep=0pt]

\newtheorem{theorem}{Theorem}[section]
\newtheorem{lemma}[theorem]{Lemma}

\newtheorem*{theorem*}{Theorem}
\newtheorem*{ass*}{Assumption}

\newtheorem{proposition}[theorem]{Proposition}
\newtheorem{question}[theorem]{Question}

\newcommand{\D}{\mathfrak{D}}

\numberwithin{equation}{section}

\begin{document}

	\title[On The Index of  trinomial over Valued Fields]{On The Index of generalized trinomial over Valued Fields}

	\author[V.~Adithya]{V. Adithya}

	\address[Adithya]{Indian Statistical Institute, New Delhi, India}
	\email{vadithya290102@gmail.com}

\begin{abstract}
Let $\nu$ be a Krull valuation of arbitrary rank on a field with valuation ring $R_\nu$, and let $\theta$ be a root of the irreducible polynomial $F(x)=(x^k+c)^m-ax^n\in R_\nu[x]$ and $1\leq n<km$. We establish necessary and sufficient conditions for the integral closedness of $R_\nu[\theta]$, expressed explicitly in terms of the coefficients $a$, $b$, and the integers $m$, $n$, and $k$. In particular, when $\nu$ is the $p$-adic valuation on $\mathbb{Q}$, our results yield criteria for determining the primes dividing the index $[\mathbb{Z}_K:\mathbb{Z}[\theta]],$
where $K=\mathbb{Q}(\theta)$ and $\mathbb{Z}_K$ denotes the ring of integers of $K$.
\end{abstract}
\maketitle

\section{Introduction}
Determining when a simple ring extension $R[\theta]$ is integrally closed in its quotient field is a classical problem in algebraic number theory. A major motivation for this problem comes from the study of monogenic number fields. Recall that an algebraic number field $K$ with ring of integers $\mathbb{Z}_K$ is called \textit{monogenic} if it admits a power integral basis, meaning $\mathbb{Z}_K = \mathbb{Z}[\theta]$ for some algebraic integer $\theta$. This global property can be checked locally: $\mathbb{Z}_K = \mathbb{Z}[\theta]$ if and only if the localization $\mathbb{Z}_{(p)}[\theta]$ is integrally closed for every prime $p$. In 1878, Dedekind provided a necessary and sufficient criterion to determine whether $\mathbb{Z}_{(p)}[\theta]$ is integrally closed, based on the factorization of the minimal polynomial of $\theta$ modulo $p$.

Over the years, Dedekind's criterion has been extended to more general settings. In 2006, Ershov~\cite{MR2455373} generalized the criterion to arbitrary valuation rings. Later, in 2016, Khanduja and Jhorar~\cite{KJ} established a highly useful equivalent formulation of this Generalized Dedekind Criterion. They showed that if $v$ is a Krull valuation on a field with valuation ring $R_v$ and maximal ideal $M_v$, then the simple extension $R_v[\theta]$ is integrally closed if and only if the minimal polynomial $F(x)$ does not belong to the ideal $\langle M_v, g_i(x) \rangle^2$ for any monic polynomial $g_i(x) \in R_v[x]$ that is irreducible and a factor of $F(x)$ modulo $M_v$.

These criteria have proved to be powerful tools for investigating the monogenity of specific classes of polynomials. For example,  in 2017, Jakhar, Khanduja, and Sangwan~\cite{MR3713088} established necessary and sufficient conditions for the monogenity of the trinomial $x^n+ax^k+b$. The following year, the same authors~\cite{MR3720858} used valuation-theoretic methods to give explicit criteria for the integral closedness of $R_v[\theta]$ when $\theta$ is a root of $x^n+ax^m+b$. Since then, several authors have studied the monogenity of generalized trinomials and their compositions \cite{MR4925475, MR5046175, MR4898732, J-K-Y, MR4828310}. More recently, Barman \textit{et al.}~\cite{BARMAN2026108328} gave necessary and sufficient conditions for the monogenity of the polynomial $(x^k+c)^m - ax^n$ over $\mathbb{Z}$. For a broader overview of recent developments in the theory of monogenic number fields, we refer to the survey articles by Ga\'al~\cite{Gaal2024} and Evertse and Gy\H{o}ry~\cite{EvertseGyory2025}.

These developments naturally lead to the following question:

\begin{question}
    Under what conditions is the simple extension $R_v[\theta]$ integrally closed when $\theta$ is a root of a given irreducible polynomial $f(x) \in R_v[x]$?
\end{question}

In this paper, we extend the results of Barman \textit{et al.}~\cite{BARMAN2026108328} by generalizing the arithmetic properties of the family $F(x) = (x^k+c)^m - ax^n$ to the arbitrary valuation ring setting. Let $v$ be a Krull valuation of arbitrary rank on a field with valuation ring $R_v$, and let $\theta$ be a root of the irreducible polynomial $F(x) \in R_v[x]$ with $km > n \ge 1$. We establish explicit necessary and sufficient conditions for $R_v[\theta]$ to be integrally closed in terms of the coefficients $a, c$ and the integers $k, m, n$. Since the reduction of these polynomials modulo $M_v$ often yields repeated roots, our proofs rely on a careful analysis of the discriminant and modular factorizations using the equivalent condition of Dedekind criterion of provided by  Khanduja and Jhorar.
As an immediate consequence of our main result Theorem~\ref{monogeneity of f(x)}, we show how this general criterion specializes to the $p$-adic valuation on $\mathbb{Q}$. This explicitly recovers the criteria for determining the primes that divide the index $[\mathbb{Z}_K : \mathbb{Z}[\theta]]$. The discriminant of $F(x)=(x^k+c)^m-ax^n$ was provided by Barman \textit{et al.} \cite{BARMAN2026108328}, which is as follows:
\begin{equation}
    \D_F
=(-1)^{\binom{km}{2}+(km+n+t)(m-1)}\,a^{\,k(m-1)}c^{m(n-1)}\,
\left[\,(km)^{\,k_1m}c^{k_1m-n_1} - a^{k_1}n^{\,n_1}\,(km-n)^{\,k_1m-n_1}\right]^{t}.
\end{equation}
where $n = n_1 t$, $k = k_1 t$ and $t=\gcd(n,k)$. We shall denote by $\mathcal{C},\, \mathcal E$, the integers defined by
\begin{equation}\label{CE}
	\mathcal C=(km)^{\,k_1m}c^{k_1m-n_1}  ~{\rm and}  ~\mathcal E= a^{k_1}n^{\,n_1}\,(km-n)^{\,k_1m-n_1},
\end{equation}
then in view of the above theorem $\D_F = (-1)^{\binom{km}{2}+(km+n+t)(m-1)}\,a^{\,k(m-1)}c^{m(n-1)}\,(\mathcal C-\mathcal E)^{t}.$

\section{main result}
\begin{theorem} \label{monogeneity of f(x)}
    Let $v$ be a Krull valuation of arbitrary rank on a field having valuation ring $R_v$, maximal ideal $M_v$, and perfect residue field. Let $p$ denote the characteristic of the residue field $R_v/M_v$, if it is positive. Let $\theta$ be a root of the irreducible polynomial $F(x) = (x^k+c)^m - ax^n \in R_v[x]$ with $km > n \ge 1$. Let $\gcd(n,k) = t$ and $n = n_1t,\, k = k_1t$. Let $\D_F$ be the discriminant of $F(x)$. Under the assumption $v(\D_F) > 0$, the ring extension $R_v[\theta]$ is integrally closed if and only if $M_v$ is a principal ideal, say generated by $\pi$, and one of the following conditions is satisfied:
    \begin{enumerate}[label=\textup{(\roman{enumi})}]
        \item When $c \in M_v$ and $a \in M_v$, then $m=1$ and $v(c) = v(\pi)$;
        
        \item When $n=1$, $c \in M_v$ and $a \notin M_v$ with $j \ge 1$ as the highest power of $p$ dividing $km-1$, then either $v(c_2) \ge v(\pi)$ and $v(a_1) = 0$ or $v(c_2(a^k a_1^{km-1} - (-m c_2 a)^{km-1})) = 0$, where $c_2 = \frac{c}{\pi}$ and $a_1 = \frac{ (a')^{p^j}-a}{\pi}$ with $a' \in R_v$ such that $(\overline{a'})^{p^j} = \overline{a}$;
        
        \item When $n \ge 2$, $c \in M_v$ and $a \notin M_v$ with $\ell \ge 0$ as the highest power of $p$ dividing $k-n$, then $m=1$ and 
        \begin{itemize}
            \item[\textup{(a)}] If $\ell=0$, then $v(c) = v(\pi)$;
            \item[\textup{(b)}] If $\ell \ge 1$, then either $v(a_1) \ge v(\pi)$ and $v(c_2) = 0$, or $v\bigl(a_1 [a^{n_1} a_1^{k_1-n_1} - (-c_2)^{k_1-n_1}]\bigr) = 0$, where $c_2 = \frac{c}{\pi}$ and $a_1 = \frac{(a')^{p^\ell}-a}{\pi}$ with $a' \in R_v$ such that $(\overline{a'})^{p^\ell} = \overline{a}$;
        \end{itemize}
        
        \item When $c \notin M_v$ and $a \in M_v$ and further:
        \begin{itemize}
            \item[\textup{(a)}] If $m \ge 2$, then $v(a) = v(\pi)$;
            \item[\textup{(b)}] If $m = 1$ with $\ell \ge 1$ as the highest power of $p$ dividing $k$, then either $v(a_2) \ge v(\pi)$ and $v(c_1) = 0$ or $v\bigl(a_2 [a_2^{k_1}(-c)^{n_1} - c_1^{k_1}]\bigr) = 0$, where $a_2 = \frac{a}{\pi}$ and $c_1 = \frac{c + (-c')^{p^\ell}}{\pi}$ with $c' \in R_v$ such that $(\overline{c'})^{p^\ell} = \overline{c}$;
        \end{itemize}
        
        \item When $ca \notin M_v$ and $k \in M_v$, where $k = s' p^j$ and $n = s p^j$ with $p \nmid \gcd(s,s')$, then the polynomials $\frac{1}{\pi}[a_1 x^n - \pi m t(x)(x^k + c)^{m-1}]$ and $(x^{s'} + c')^m - \overline{a} x^s$ are coprime modulo $M_v$, where $a_1 = \frac{a - (a')^{p^j}}{\pi}$, $t(x) =  \frac{1}{\pi} \left( \sum_{i=1}^{p^{j}-1} \binom{p^{j}}{i} x^{s' i} (c')^{p^{j}-i} + (c')^{p^j} - c \right) $, and $a', c' \in R_v$ such that $(\overline{a'})^{p^j} = \overline{a}$ and $(\overline{c'})^{p^j} = \overline{c}$;
        
        \item When $cak \notin M_v$ and $m \in M_v$, then $n \in M_v$ and $v(a - (a')^p) = v(\pi)$, where $a' \in R_v$ such that $(\overline{a'})^p = \overline{a}$;
        
        \item When $cakm \notin M_v$, then $\mathcal{C} - \mathcal{E} \notin M_v^2$, where $\mathcal{C}$ and $\mathcal{E}$ are as defined in \ref{CE}.
    \end{enumerate}
\end{theorem}
Let $\nu$ denote the $p$-adic valuation on $\mathbb{Q}$, and let $\theta$ be a root of the irreducible polynomial $f(x)=(x^k+c)^m-ax^n\in\mathbb{Z}[x].$ Applying Theorem~\ref{monogeneity of f(x)}, together with Fermat's little theorem, shows that the localized ring $\mathbb{Z}_{(p)}[\theta]$ is integrally closed in $K=\mathbb{Q}(\theta)$ if and only if one of the conditions stated above holds. Since
\[
\mathbb{Z}_{(p)}[\theta]\text{ is integrally closed}
\quad\Longleftrightarrow\quad
p\nmid[\mathbb{Z}_K:\mathbb{Z}[\theta]],
\]
the corresponding criterion for the index follows immediately. Thus, our general result specializes to the $p$-adic setting and recovers both the  Theorem ~2.2 of~\cite{BARMAN2026108328} as a special case, and the main result of \cite{MR3720858} by taking the following values $m=1,a=-1$ in $F(x)$. 

\section{preliminaries}
Throughout the paper, $\nu$ denotes a Krull valuation of arbitrary rank of a field with valuation ring $R_{\nu}$ and maximal ideal $M_{\nu}$. Given an element $\alpha \in R_{\nu}$, we denote its image under the canonical homomorphism from $R_{\nu}$ onto $R_{\nu}/M_{\nu}$ by $\bar{\alpha}$. Similarly, for a polynomial $g(x)\in R_{\nu}[x]$, $\bar{g}(x)$ is the polynomial over $R_{\nu}/M_{\nu}$ obtained by replacing each coefficient of $g(x)$ with its image modulo $M_{\nu}$.
The following observation can be easily deduced using Binomial expansion.
\begin{lemma}\label{binomial expansion} 
    Let $v$, $R_v, M_v$ defined as above.  Let the residue field $R_v/M_v$ be perfect with prime characteristic $p > 0$ and suppose that the maximal ideal is principal, say $M_v = \langle \pi \rangle$. Let $j \geq 1$ be the highest power of $p$ dividing $n=p^{j}s'$, and let $c \in R_v \setminus M_v$. Choose $c' \in R_v$ such that $(\overline{c'})^{\,p^j} = \overline{c}$. If $\overline{g}_1(x)\cdots \overline{g}_r(x)$ is the factorization of $x^{s'}-\overline{c'}$ into a product of distinct monic irreducible polynomials over $R_v/M_v$ with $g_i(x)\in R_v[x]$ monic, then 
    $$x^n-c = (g_1(x)\cdots g_r(x) + \pi H_1(x))^{p^j} + p \, g_1(x)\cdots g_r(x)H_2(x) + p \pi H_3(x) + (c')^{p^j} - c$$ 
    for some polynomials $H_1(x), H_2(x), H_3(x)\in R_v[x]$.
\end{lemma}
 In 2010, Khanduja and
Kumar~\cite{MR2598906} generalized Dedekind criterion to general valuation rings of Krull valuations by
reducing the problem to Henselian valuation rings which is as follows:
\begin{theorem}[Generalized Dedekind Criterion (GDC)]
Let $v$, $R_v, M_v$ defined as above. Let $F(x)\in R_\nu[x]$ be a monic irreducible polynomial having a root
$\theta$ in its splitting field, and suppose that
\[
\overline{F}(x)
=
\overline{g}_1(x)^{e_1}\cdots
\overline{g}_t(x)^{e_t}
\]
is the factorization of $\overline{F}(x)$ into powers of distinct irreducible
polynomials over $R_\nu/M_\nu$, where each $g_i(x)\in R_\nu[x]$ is monic.
Then $R_\nu[\theta]$ is integrally closed if and only if one of the following
holds:

\begin{enumerate}
\item[(i)]
$e_i=1$ for every $1\le i\le t$;

\item[(ii)]
There exists a unique index $1\le j\le t$ such that $e_j>1$.
In this case, $M_\nu$ is a principal ideal, say
$M_\nu=(\pi)$, and
$\overline{g}_j(x)$ does not divide
$\overline{Y}(x)$, where
\[
Y(x)
=
\frac{1}{\pi}
\Bigl(
F(x)-g_1(x)^{e_1}\cdots g_t(x)^{e_t}
\Bigr)
\in R_\nu[x].
\]
\end{enumerate}
\end{theorem}
\noindent In 2016,  the following equivalent version of the above criterion was proved (see \cite{KJ}).
\begin{theorem}\label{valuation dede}
Let $v$, $R_v$, $M_v$, $\theta$, $F(x)$, and $g_i(x)$ be as in the previous theorem. Then $R_v[\theta]$ is integrally closed if and only if $F(x) \notin \langle M_v, g_i(x)\rangle ^2 \ \text{in } R_v[x] \text{ for any } i,\ 1 \le i \le t.$
\end{theorem}

\begin{lemma}\label{lemma:3.6}
    Let $v$, $R_v, M_v$ defined as above.  Let the residue field $R_v/M_v$ be perfect with prime characteristic $p > 0$ and suppose that the maximal ideal is principal $M_v = \langle \pi \rangle$. Let $F(x)=(x^k+c)^m-ax^n \in R_v[x]$ be a polynomial of degree $km$. Suppose $n=p^js$ and $k=p^js'$, where $p\nmid\gcd(s,s')$. Since the residue field is perfect, choose elements $a', c' \in R_v$ such that $(\overline{a'})^{p^j} = \overline{a}$ and $(\overline{c'})^{p^j} = \overline{c}$. Define $h(x)=(x^{s'}+c')^m-a'x^s \in R_v[x]$. Then $F(x)$ can be expressed as
    $$F(x) =  h(x)^{p^{j}}+((a')^{p^{j}} - a)x^{n} - \pi m\, t(x)(x^{k} + c)^{m-1}+ \pi\,h(x)h_{1}(x) + \pi^{2}h_{4}(x)$$
    for some $h_1(x), h_4(x) \in R_v[x]$ and 
    $$t(x) = \frac{1}{\pi} \left( \sum_{i=1}^{p^{j}-1} \binom{p^{j}}{i} x^{s' i} (c')^{p^{j}-i} + (c')^{p^j} - c \right) \in R_v[x].$$
\end{lemma}

\begin{proof}
    By definition, $h(x) = (x^{s'} + c')^m - a' x^s$. We first expand $(x^{s'} + c')^{mp^j}$ using the binomial expansion:
    $$(x^{s'} + c')^{mp^j} = \bigl( h(x) + a'x^s \bigr)^{p^{j}} = h(x)^{p^{j}} + \pi\,h(x)h_{1}(x)  + \pi^{2}h_{2}(x)+ (a')^{p^{j}}x^{n}$$
    for some $h_{1}(x), h_{2}(x) \in R_v[x]$. \\
    On the other hand, we can rewrite the base of the power algebraically. We have:
    $$(x^{s'} + c')^{p^j} = x^{s'p^j} + (c')^{p^j} + \sum_{i=1}^{p^j-1} \binom{p^j}{i} x^{s' i} (c')^{p^j-i}.$$
    Noting that $p \in M_v = \langle \pi \rangle$, the binomial coefficients $\binom{p^j}{i}$ are divisible by $\pi$ for $1 \le i \le p^j-1$. 
    Since  the residue field is perfect, our choice of $c'$ ensures that $(\overline{c'})^{p^j} = \overline{c}$, which means $(c')^{p^j} - c \in \langle \pi \rangle$. We can therefore factor $\pi$ out of the remaining terms:
    $$(x^{s'} + c')^{p^j} = x^{k} + c + \pi\,t(x),$$
    where
    $$t(x) = \frac{1}{\pi} \left( \sum_{i=1}^{p^{j}-1} \binom{p^{j}}{i} x^{s' i} (c')^{p^{j}-i} + (c')^{p^j} - c \right) \in R_v[x].$$
    Expanding this expression to the power of $m$ using the binomial expansion yields:
    $$\bigl( x^{k} + c + \pi\,t(x) \bigr)^{m} = (x^{k} + c)^{m} + \pi m\, t(x) (x^{k} + c)^{m-1} + \pi^{2} h_{3}(x)$$
    for some $h_{3}(x) \in R_v[x]$. Comparing the two expansions for $(x^{s'} + c')^{mp^j}$ gives:
    $$h(x)^{p^{j}} + \pi\,h(x)h_{1}(x) + \pi^{2}h_{2}(x)+ (a')^{p^{j}}x^{n} = (x^{k} + c)^{m} + \pi m\, t(x) (x^{k} + c)^{m-1} + \pi^{2} h_{3}(x).$$
    Rearranging the terms, we obtain
    $$F(x) =  h(x)^{p^{j}}+((a')^{p^{j}} - a)x^{n} - \pi m\, t(x)(x^{k} + c)^{m-1}+ \pi\,h(x)h_{1}(x) + \pi^{2}h_{4}(x),$$
    where $h_{4}(x) = h_2(x) - h_3(x) \in R_v[x]$.
\end{proof}

{}

\begin{proposition}\label{2.06}
    Let $F(x), ~n_1, ~k_1, ~\mathcal{C}$ and $\mathcal{E}$ be as defined in Theorem \ref{monogeneity of f(x)}. Let $u_1, u_2\in \mathbb{Z}$ satisfy $ku_1+nu_2=t$. Assume that $v(ackm(km-n)) = 0$. Furthermore, suppose that there exist elements $\alpha_1, \alpha_2 \in R_v$ such that $\alpha_1 \equiv \frac{nc}{km-n} \pmod{M_v^2}$ and $\alpha_2\equiv a^{-1} \left(\frac{kmc}{km-n}\right)^m \pmod{M_v^2}$. Then the following are equivalent for $i=1,2$.
    \begin{itemize}
        \item[\textup{(i)}] $\mathcal{C}-\mathcal{E} \in M_v^i$;
        \item [\textup{(ii)}] $\alpha_2^{k_1}\equiv\alpha_1^{n_1}\pmod{M_v^i}$;
        \item [\textup{(iii)}] If $\beta \in R_v$ satisfies $\beta^t \equiv \alpha_1^{u_1} \alpha_2^{u_2} \pmod{M_v^i}$, then $\beta$ is a repeated root of $F(x)$ modulo $M_v^i$.
    \end{itemize}
\end{proposition}

\begin{proof}
    We provide the proof of the equivalence for $i=1$. A similar argument will prove the case for $i=2$. Using $v(ackm(km-n)) = 0$, it is easy to note that $\alpha_1 \alpha_2 \not \equiv 0 \pmod{M_v}$.\\
\noindent (i) $\Leftrightarrow$ (ii) Since $\mathcal{C}-\mathcal{E} \in M_v$, we have
    \begin{align*}
        (km)^{k_1m} c^{k_1m-n_1} \equiv a^{k_1} n^{n_1} (km-n)^{k_1m-n_1} \pmod{M_v}.
    \end{align*}
    Since $v(ac(km-n)) = 0$, the above equation can be rewritten as
    \begin{equation}
        \left( \frac{nc}{km-n} \right)^{n_1} \equiv \left( a^{-1} \left( \frac{kmc}{km-n} \right)^m \right)^{k_1} \pmod{M_v}.
    \end{equation}
    Substituting the values of $\alpha_1$ and $\alpha_2$ in the above equation, we get
    \begin{align*}
        \alpha_1^{n_1}\equiv\alpha_2^{k_1}\pmod{M_v}.
    \end{align*}
    Note that all the preceding arguments are reversible, so the converse is also true.\\
    \noindent (ii) $\Rightarrow$ (iii) We show that $\beta$ with $\beta^t\equiv\alpha_1^{u_1}\alpha_2^{u_2}\pmod {M_v}$ is a repeated root of $F(x)$ modulo $M_v$. It is easy to observe that here $\beta \not\equiv 0 \pmod{M_v}$.
    First, we show that $F(\beta)\equiv 0\pmod{M_v}$.
    \begin{align*}
        F(\beta)&=(\beta^k+c)^{m}-a\beta^n,\\
        &=\left((\beta^t)^{k_1}+c\right)^{m} -a(\beta^t)^{n_1}.
    \end{align*}
    Using $\beta^t\equiv \alpha_1^{u_1}\alpha_2^{u_2}\pmod{M_v}$ in the above equation, we have
    \begin{equation}\label{.002}
        F(\beta)\equiv \left((\alpha_1^{u_1}\alpha_2^{u_2})^{k_1}+c\right)^{m} -a(\alpha_1^{u_1}\alpha_2^{u_2})^{n_1} \pmod{M_v}.
    \end{equation}
    Using $k_1u_1 + n_1u_2 = 1$ and $\alpha_2^{k_1}\equiv\alpha_1^{n_1}\pmod{M_v}$ in \eqref{.002}, we get
    \begin{equation*}
        F(\beta) \equiv (\alpha_1+c)^{m} -a \alpha_2\pmod{M_v}.
    \end{equation*}
    Substituting the values of $\alpha_1$ and $\alpha_2$ in the above equation, we see that
    \begin{align*}
        F(\beta)&\equiv \left(\frac{nc}{km-n} + c\right)^m -\left(\frac{kmc}{km-n}\right)^m\pmod{M_v},\\
        &\equiv0\pmod{M_v}.
    \end{align*}
    Now, we show that $F'(\beta)\equiv 0 \pmod{M_v}$. Taking the derivative of $F(x)$ and substituting $x=\beta$, we obtain
    \begin{align*}
        F'(\beta)&\equiv km\beta^{k-1}(\beta^{k}+c)^{m-1} - an\beta^{n-1}\pmod{M_v},\\
        &\equiv \beta^{-1} \left(km\beta^{k_1t}(\beta^{k_1t}+c)^{m-1} - an\beta^{n_1t} \right)\pmod{M_v},\\
        &\equiv \beta^{-1} \left(km(\alpha_1^{u_1}\alpha_2^{u_2})^{k_1}((\alpha_1^{u_1}\alpha_2^{u_2})^{k_1}+c)^{m-1} -an(\alpha_1^{u_1}\alpha_2^{u_2})^{n_1}\right)\pmod{M_v}.
    \end{align*}
    Using $n_1u_2+k_1u_1=1$ and $\alpha_{1}^{n_1}\equiv\alpha_2^{k_1}\pmod{M_v}$ in the above equation, we have
    \begin{align*}
        F'(\beta) \equiv \beta^{-1} \left(km\alpha_{1}(\alpha_{1}+c)^{m-1} - n a\alpha_2 \right)\pmod{M_v}.
    \end{align*}		
    Substituting values of $\alpha_1$ and $\alpha_{2}$ in the above equation, we obtain
    \begin{align}
        F'(\beta)&\equiv \beta^{-1} \left( \frac{kmnc}{km-n} \left( \frac{nc}{km-n} + c \right)^{m-1} -n \left( \frac{kmc}{km-n} \right)^m \right) \pmod{M_v},\nonumber\\
        &\equiv \beta^{-1} \left(n\frac{kmc}{km-n} \left(\frac{kmc}{km-n}\right)^{m-1} - n \left(\frac{kmc}{km-n}\right)^{m}\right) \pmod{M_v},\nonumber\\
        &\equiv 0 \pmod{M_v}.
    \end{align}
    
    \noindent (iii) $\Rightarrow$ (ii)
    Suppose that $\beta$ is any repeated root of $\overline{F}(x)= (x^k+\overline{c})^m - \overline{a} x^n$ in the algebraic closure of $R_v/M_v$. Note that $\beta \neq \overline{0}$ as $v(c) = 0$. Then
    \begin{equation}\label {0.008}
        \overline{F}(\beta) = ({\beta}^k+\overline{c})^m -\overline{a} {\beta}^n = \overline{0};~~ \overline{F}^{\prime}(\beta)= \overline{k} \overline{m} \beta^{k-1}({\beta}^k+\overline{c})^{m-1} -\overline{a} \overline{n}{\beta}^{n-1} = \overline{0}. 
    \end{equation}
    On substituting $\overline{a} {\beta}^{n} = ({\beta}^k+\overline{c})^m$ into \eqref{0.008}, we see that
    \begin{equation}
        \frac{1}{\beta}(\beta^k+c)^{m-1} \left[ (km-n) \beta^k - nc\right] \equiv 0 \pmod{M_v}.
    \end{equation}
    Observe that $(\beta^k+c) \not\equiv 0\pmod{M_v}$, otherwise in view of the first equation of \eqref{0.008}, we would obtain $a\beta \equiv 0\pmod{M_v}$, which is not possible. Therefore, keeping in mind that $v(km-n) = 0$, we have
    \begin{align}\label{0.009}
        \beta^k &\equiv \frac{nc}{km-n}\pmod{M_v},\nonumber\\
        &\equiv \alpha_1 \pmod{M_v}.
    \end{align}
    Changing $k$ by $tk_1$ and substituting the value $\beta^t \equiv \alpha_1^{u_1}\alpha_2^{u_2}\pmod{M_v}$ into \eqref{0.009}, we have
    \begin{align*}
        (\alpha_1^{u_1}\alpha_2^{u_2})^{k_1} \equiv \alpha_1 \pmod{M_v}.
    \end{align*}
    Taking the power $n_1$ on both sides of the above equation, we obtain
    \begin{equation}\label{.a}
        (\alpha_1^{u_1}\alpha_2^{u_2})^{k_1n_1} \equiv \alpha_1^{n_1} \pmod{M_v}.
    \end{equation}
    Substituting $\beta^k \equiv \frac{nc}{km-n}\pmod{M_v}$ into the first equation of \eqref{0.008}, we get
    \begin{align}\label{special}
        \beta^n & \equiv a^{-1} \left(\frac{kmc}{km-n}\right)^m \pmod{M_v},\\\nonumber
        &\equiv \alpha_2 \pmod{M_v}.
    \end{align}
    Changing $n$ by $n_1t$ and substituting $\beta^t \equiv \alpha_1^{u_1}\alpha_2^{u_2} \pmod{M_v}$, we observe
    \begin{align*}
        (\alpha_1^{u_1}\alpha_2^{u_2})^{n_1} \equiv \alpha_2 \pmod{M_v}.
    \end{align*}
    Taking the power $k_1$ on both sides of the above equation, we get
    \begin{equation}\label{.b}
        (\alpha_1^{u_1}\alpha_2^{u_2})^{n_1k_1} \equiv \alpha_2^{k_1} \pmod{M_v}.
    \end{equation}
    From \eqref{.a} and \eqref{.b}, we conclude that $\alpha_1^{n_1} \equiv \alpha_2^{k_1} \pmod{M_v}$. This completes the proof of the proposition.
\end{proof}
\section{Proof of  Theorem~\ref{monogeneity of f(x)}}
\begin{proof}
Let $v$ be a Krull valuation and assume that $v(\D_F) > 0$. In view of Theorem~\ref{valuation dede}, the ring $R_v[\theta]$ is integrally closed if and only if $F(x) \notin \langle M_v, g(x) \rangle^2$ for every monic polynomial $g(x) \in R_v[x]$ whose reduction $\overline{g}(x)$ is a repeated irreducible factor of $\overline{F}(x)$ modulo $M_v$. Note that since the maximal ideal $M_v$ is generated by $\pi$, this is equivalent to checking $F(x) \notin \langle \pi, g(x) \rangle^2$.\\[2mm]
\textbf{Case (i).} Suppose $c \in M_v$ and $a \in M_v$. In this case, $F(x) = (x^k + c)^m - a x^n \equiv x^{km} \pmod{M_v}.$ The only irreducible factor of $\overline{F}(x)$ is $x$. It is easily checked that $F(x) \in \langle \pi, x \rangle^2$ if and only if $c^m \in M_v^2$. 
Therefore, by Theorem~\ref{valuation dede}, $R_v[\theta]$ is integrally closed if and only if $F(x) \notin \langle \pi, x \rangle^2$, which holds if and only if $c^m \notin M_v^2$. Since $c \in M_v$, this condition can hold only when $m=1$ and $c \notin M_v^2$ (i.e., $v(c) = v(\pi)$). This completes the proof in this case.\\[1mm]
\textbf{Case (ii).} Suppose $n=1$, $c \in M_v$, and $a \notin M_v$. Since $v(\D_F) > 0$, it follows from the discriminant formula that $v(km-1) > 0$. Let $p$ be the characteristic of $R_v/M_v$. Write $km-1=p^{j}s$ with $p\nmid s$. Choose $a' \in R_v$ such that $(\overline{a'})^{\,p^j} = \overline{a}$. In this situation, 
$$F(x)\equiv x\bigl(x^{km-1}-a\bigr) \equiv x\bigl(x^{s}-a'\bigr)^{p^{j}} \pmod{M_v}.$$
Set $h(x)=x^{s}-a'$ ,i.e., $x^{s}=h(x)+a'$. Raising both sides to the power $p^{j}$ and using the binomial expansion, we obtain:
\begin{equation}\label{eq:4.9}
x^{km-1} = h(x)^{p^{j}} + (a')^{p^{j}} + \pi h(x)H_{1}(x)
\end{equation}
for some polynomial $H_{1}(x)\in R_v[x]$. Since $c \in M_v$, we may expand $F(x)$ as
\begin{equation}\label{eq:4.10}
F(x)= x(x^{km-1}-a)+m c x^{k(m-1)} + \pi^{2}H_{2}(x)
\end{equation}
for some $H_{2}(x)\in R_v[x]$. Write $h(x)=g_{1}(x)\cdots g_{t}(x)+\pi H(x),$ where $g_{1}(x),\ldots,g_{t}(x)$ are distinct monic polynomials that are irreducible modulo $M_v$ and $H(x)\in R_v[x]$.
Substituting \eqref{eq:4.9} into \eqref{eq:4.10} and using the expression for $h(x)$, we obtain:
\begin{align*}
F(x) &= x\,h(x)^{p^{j}} + m c x^{k(m-1)} + ((a')^{p^{j}}-a)x + \pi h(x)H_{3}(x) + \pi^{2}H_{4}(x), \\
&= x\left(\prod_{i=1}^{t} g_i(x)+\pi H(x)\right)^{p^{j}} + x\bigl(m c x^{k(m-1)-1} + ((a')^{p^{j}}-a)\bigr) + \pi h(x)H_{3}(x) + \pi^{2}H_{4}(x)
\end{align*}
for some $H_{3}(x), H_{4}(x)\in R_v[x]$. Observe that $x, \overline{g}_{1}(x), \ldots, \overline{g}_{t}(x)$ are pairwise distinct irreducible factors of $\overline{F}(x)$.
Since $j\ge1$, it follows that $F(x)\in\langle \pi, g_i(x)\rangle^{2}$ for some $i$ if and only if the term $x\bigl(m c x^{k(m-1)-1} + ((a')^{p^{j}}-a)\bigr)$ lies in $\langle \pi, g_i(x)\rangle^{2}$. Since $x$ and $g_i(x)$ are coprime modulo $M_v$, this is equivalent to $m c x^{k(m-1)-1} + ((a')^{p^{j}}-a) \in \langle \pi, g_i(x) \rangle^2$. Write $(a')^{p^{j}} - a = \pi a_{1}$ and $c = \pi c_{2}.$ Since $v(km-1) > 0$ implies $v(m) = 0$, $F(x)\in\langle \pi, g_i(x)\rangle^{2}$ for some $i$ if and only if either $v(c_{2}) \ge v(\pi)$ and $v(a_{1}) \ge v(\pi)$, or $v(c_{2}) = 0$ and the polynomials $\overline{m}\,\overline{c}_{2}x^{k(m-1)-1} + \overline{a}_{1}$ and $x^{s}-\overline{a'}$ have a common root modulo $M_v$. Let $\eta$ be such a common root. Then $m c_{2}\eta^{k(m-1)-1} = - a_{1}$ and $\eta^{km-1} = a$ modulo $M_v$. Eliminating $\eta$, we obtain the common-root condition $a^k a_1^{km-1} = (-m c_2 a)^{km-1}$. Therefore, we conclude that $R_v[\theta]$ is integrally closed if and only if either $v(c_2) \ge v(\pi)$ and $v(a_1) = 0$ or $v(c_2(a^ka_1^{km-1}-(-mc_2a)^{km-1})) = 0$. This completes the proof in the present case.\\[2mm]
\textbf{Case (iii).} Suppose $n \geq 2$, $c \in M_v$, and $a \notin M_v$. In this case,
$F(x)\equiv x^{n}\bigl(x^{km-n}-a\bigr)\pmod{M_v}.$ Since $n\ge2$, it follows that $x$ is a repeated root of $\overline{F}(x)$. Arguing as in Case (i), we obtain $F(x) \notin \langle \pi, x \rangle^2$ if and only if $c^m \notin M_v^2$. As $c \in M_v$, this is possible only when $m=1$ and $c \notin M_v^2$ (i.e., $v(c) = v(\pi)$). Hence, we may assume throughout this case that $m=1$. Then $F(x)=x^{k}-ax^{n}+c$ and modulo $M_v$ we have $\overline{F}(x)= x^{n}\bigl(x^{k-n}-\overline{a}\bigr).$ Let $\ell\ge0$ be the highest power of $p$ dividing $k-n$.\\
First, assume that $\ell=0$, i.e., $v(k-n) = 0$.
Since $c \in M_v$ and $a \notin M_v$, the reduction $\overline{F}(x)=x^{n}\bigl(x^{k-n}-\overline{a}\bigr)$ has $x$ as its only possible repeated irreducible factor. In this situation, it is easy to check that $F(x) \in \langle \pi, x \rangle^2$ if and only if $c \in M_v^2$. Thus, when $\ell=0$, the desired conclusion $v(c) = v(\pi)$ follows immediately.\\
Now assume that $\ell\ge1$ and write $k-n=p^{\ell}s$ with $p\nmid s$. Set $h(x)=x^{s}-a'$ where $(\overline{a'})^{\,p^\ell} = \overline{a}$. We may write $h(x)=g_{1}(x)\cdots g_{t}(x)+\pi H(x),$ where $g_{1}(x),\dots,g_{t}(x)$ are distinct monic polynomials that are irreducible modulo $M_v$ and $H(x)\in R_v[x]$.
Applying Lemma~\ref{binomial expansion} to $h(x)$, we obtain
\[
F(x)=x^{n}\Biggl[
\left(\prod_{i=1}^{t} g_i(x)+\pi H(x)\right)^{p^{\ell}}
+ \pi H_1(x)\prod_{i=1}^{t} g_i(x)
+ \pi^{2} H_2(x)
-a+(a')^{p^{\ell}}
\Biggr]+c
\]
for some polynomials $H_1(x), H_2(x)\in R_v[x]$. Observe that $x,\overline{g}_{1}(x),\ldots,\overline{g}_{t}(x)$ are pairwise distinct irreducible factors of $\overline{F}(x)$.
Since $\ell\ge1$, the first three terms inside the brackets are in $\langle \pi, g_i(x)\rangle^{2}$ for each $1\le i\le t$.
Consequently, $F(x)\in\langle \pi, g_i(x)\rangle^{2}$ for some $i$ if and only if the remaining term $\bigl((a')^{p^{\ell}}-a\bigr)x^{n}+c$ lies in $\langle \pi, g_i(x)\rangle^{2}$. Write $(a')^{p^{\ell}} - a=\pi a_{1}$ and $c=\pi c_{2}.$
Then $F(x)\in\langle \pi, g_i(x)\rangle^{2}$ for some $i$ if and only if either $v(a_{1}) \ge v(\pi)$ and $v(c_{2}) \ge v(\pi)$, or $v(a_1) = 0$ and the polynomials $\overline{a}_{1}x^{n}+\overline{c}_{2}$ and $x^{k-n}-\overline{a}$ have a common root modulo $M_v$. Let $\eta$ be such a common root. Then
\[
\overline{a}_{1}\eta^{n} = - \overline{c}_{2}, \qquad \eta^{k}=\overline{a}\eta^{n}.
\]
Simplifying these relations and using $k=k_{1}t$ and $n=n_{1}t$, we obtain
\[
(-\overline{c}_{2})^{k_{1}-n_{1}} = \overline{a}^{n_{1}} \overline{a}_{1}^{k_{1}-n_{1}}.
\]
By Theorem~\ref{valuation dede}, we conclude that $R_v[\theta]$ is integrally closed if and only if either $v(a_{1}) \ge v(\pi)$ and $v(c_{2}) = 0$, or $v\bigl(a_1 [a^{n_{1}}a_{1}^{k_{1}-n_{1}}-(-c_{2})^{k_{1}-n_{1}}]\bigr) = 0$. This completes the proof in the present case.\\[2mm]
\textbf{Case (iv).} Suppose that $c \notin M_v$ and $a \in M_v$. Then $F(x) \equiv (x^{k}+c)^{m} \pmod{M_v}.$ There are two possibilities to consider, namely $m\ge 2$ and $m=1$.\\[1mm]
First, assume that $m\ge 2$. Let $g(x) \in R_v[x]$ be a monic polynomial whose reduction $\overline{g}(x)$ is an irreducible factor of $x^{k}+\overline{c}$ over $R_v/M_v$. Then $F(x)\in\langle \pi, g(x)\rangle^{2}$ if and only if $a \in M_v^2$. Hence, by Theorem~\ref{valuation dede}, we obtain that $R_v[\theta]$ is integrally closed if and only if $a \notin M_v^2$ (i.e., $v(a) = v(\pi)$).\\[1mm]
Now assume that $m=1$. In this case $F(x)=x^{k}-ax^{n}+c,$ and since $a \in M_v$, we have $F(x)\equiv x^{k}+c \pmod{M_v}.$ Moreover, as $v(\D_F) > 0$, it follows from the discriminant formula that $v(k) > 0$. Let $p$ be the characteristic of $R_v/M_v$. Write $k=p^{\ell}s$ with $p\nmid s$. Since $R_v/M_v$ is perfect, choose $c' \in R_v$ such that $(\overline{c'})^{\,p^\ell} = \overline{c}$. By the Binomial theorem, we then obtain $F(x)\equiv (x^{s}+c')^{\,p^{\ell}} \pmod{M_v}.$ Let $\overline{g}_{1}(x),\ldots,\overline{g}_{r}(x)$ be the factorization of $x^{s}+\overline{c'}$ over $R_v/M_v$, where $g_i(x)\in R_v[x]$ are monic, pairwise distinct, and irreducible modulo $M_v$. We may therefore write $x^{s}+c' = g_{1}(x)\cdots g_{r}(x) + \pi H(x)$ for some polynomial $H(x)\in R_v[x]$. Set $h(x)=x^{s}+c'$. Applying Lemma~\ref{binomial expansion} to $h(x)$, we obtain
\begin{equation}\label{eq:case-iii-expansion}
F(x)= \left(\prod_{i=1}^{r} g_i(x) + \pi H(x)\right)^{p^{\ell}}
   + \pi H_1(x)\prod_{i=1}^{r} g_i(x)
   + \pi^{2}H_2(x) + c +(-c')^{p^{\ell}} - ax^{n},
\end{equation} 
for some polynomials $H_1(x), H_2(x)\in R_v[x]$.
Since $\ell\ge 1$, the first three terms on the right-hand side of \eqref{eq:case-iii-expansion} lie in the ideal $\langle \pi, g_i(x)\rangle^{2}$ for each $1\le i\le r$.
Therefore, $F(x)\in\langle \pi, g_i(x)\rangle^{2}$ for some $i$ if and only if $-ax^{n}+c+(-c')^{p^{\ell}} = \pi(-a_{2}x^{n}+c_{1})$ belongs to $\langle \pi, g_i(x)\rangle^{2}$. It follows that $\pi(-a_{2}x^{n}+c_{1})\in\langle \pi, g_i(x)\rangle^{2}$ for some $i$ if and only if either $v(a_{2}) \ge v(\pi)$ and $v(c_{1}) \ge v(\pi)$, or $v(a_{2}) = 0$ and the polynomials $-\overline{a}_{2}x^{n}+\overline{c}_{1}$ and $x^{s}+\overline{c'}$ have a common root modulo $M_v$. Let $\eta$ be such a common root.
Then $-\overline{a}_{2}\eta^{n}+\overline{c}_{1}=0$ and $\eta^{k}+\overline{c}=0.$ Writing $k=k_{1}t$ and $n=n_{1}t$, we obtain
\[
\overline{a}_{2}(\eta^{t})^{n_{1}}=\overline{c}_{1}
\quad \text{and} \quad
(\eta^{t})^{k_{1}}=-\overline{c}.
\]
Raising the first equation to the power $k_1$ and the second equation to the power $n_1$, we obtain
\[
\overline{a}_{2}^{k_{1}}(\eta^{t})^{n_{1}k_{1}} =\overline{c}_{1}^{k_{1}}
\quad \text{and} \quad
(\eta^{t})^{k_{1}n_{1}} =(-\overline{c})^{n_{1}}.
\]
Comparing these expressions, we conclude that $\overline{a}_{2}^{k_{1}}(-\overline{c})^{n_{1}} = \overline{c}_{1}^{k_{1}}.$ Therefore, the polynomials $-\overline{a}_{2}x^{n}+\overline{c}_{1}$ and $x^{s}+\overline{c'}$ are coprime modulo $M_v$ if and only if $v\bigl(a_{2}^{k_{1}}(-c)^{n_{1}}-c_{1}^{k_{1}}\bigr) = 0.$ Therefore, by Theorem~\ref{valuation dede}, we obtain that $R_v[\theta]$ is integrally closed if and only if either $v(a_2) \ge v(\pi)$ and $v(c_1) = 0$ or $v\bigl(a_2(a_{2}^{k_{1}}(-c)^{n_{1}}-c_{1}^{k_{1}})\bigr) = 0$.\\[2mm]
\textbf{Case (v).} Suppose $ca \notin M_v$ and $k \in M_v$. By the discriminant formula, $v(\D_F) > 0$ implies $v(\mathcal{C} - \mathcal{E}) > 0$, and hence $v(n(km - n)) > 0.$ Since $a \notin M_v$ and $c \notin M_v$, we obtain $n \in M_v$. Write $k = p^{j}s'$ and $n = p^{j}s,$ with $j \ge 1$ and $p \nmid \gcd(s,s')$. Choose $a', c' \in R_v$ such that $(\overline{a'})^{\,p^j} = \overline{a}$ and $(\overline{c'})^{\,p^j} = \overline{c}$. Then
\[
F(x) \equiv \Bigl( (x^{s'} + c')^m - a' x^s \Bigr)^{p^{j}} \pmod{M_v}.
\]
Denote $(x^{s'} + c')^m - a' x^s$ by $h(x)$. Then one can easily check that $F(x)\equiv h(x)^{p^j} \pmod{M_v}$. Write $h(x) = g^{e_1}_1(x)\cdots g^{e_d}_d(x) + \pi H(x)$, $e_i>0$, where $g_1(x), \ldots, g_d(x)$ are monic polynomials that are pairwise distinct and irreducible modulo $M_v$ and $H(x) \in R_v[x]$. Using Lemma \ref{lemma:3.6}, we have the following:
\begin{align*}
    F(x) &=  h(x)^{\,p^{j}}+((a')^{\,p^{j}} - a)x^{n} - \pi m\, t(x)(x^{k} + c)^{\,m-1}+ \pi\,h(x)H_{1}(x) + \pi^{2}H_{4}(x),\\
    &=\left(\prod\limits_{i=1}^{d}g^{e_i}_i(x)+\pi H(x)\right)^{\,p^j}+((a')^{\,p^{j}} - a)x^{n} - \pi m\, t(x)(x^{k} + c)^{\,m-1}\\
    &\quad + \pi\left(\prod\limits_{i=1}^{d}g^{e_i}_i(x)+\pi H(x)\right)H_{1}(x) + \pi^{2}H_{4}(x)
\end{align*}
for some polynomials $H_1(x), H_4(x)\in R_v[x]$ and $t(x)$ defined in  Lemma \ref{lemma:3.6}. Write $F(x)=\left(g_1(x)^{e_1} \cdots g_d(x)^{e_d}+\pi H(x)\right)^{p^j}+\pi M(x)$ for some $M(x)\in R_v[x].$ Since $j>0$, by Theorem \ref{valuation dede}, we see that $R_v[\theta]$ is integrally closed if and only if $\overline{M}(x)$ is coprime to $\overline{h}(x)$, which holds if and only if the polynomial $\frac{1}{\pi}[((a')^{\,p^{j}} - a)x^{n} - \pi m\, t(x)(x^{k} + c)^{\,m-1}]$ is coprime to $h(x)$ modulo $M_v$. This proves the theorem in this case.\\[2mm]
\textbf{Case (vi).} Suppose $cak \notin M_v$ and $m \in M_v$. By the discriminant formula, $v(\D_F) > 0$ implies $n \in M_v$. Write $n = p^{\ell}s$ and $m = p^{\ell}s',$ where $\ell \ge 1$ and $p \nmid \gcd(s,s')$. Choose $a' \in R_v$ such that $(\overline{a'})^{\,p^\ell} = \overline{a}$. Define $h(x) = (x^{k}+c)^{s'} - a' x^{s}.$ Then it is easy to check that $F(x) \equiv h(x)^{\,p^{\ell}} \pmod{M_v}.$
We first expand $(h(x)+a' x^{s})^{p^{\ell}}$ using the Binomial expansion:
\[
(h(x)+a' x^{s})^{p^{\ell}} = h(x)^{p^{\ell}} + \pi H_1(x)h(x) + (a')^{p^{\ell}}x^{n}
\]
for some polynomial $H_1(x)\in R_v[x]$. Since $h(x)=(x^{k}+c)^{s'}-a' x^{s}$, it follows that
\[
(x^{k}+c)^{m} = h(x)^{p^{\ell}} + \pi\,h(x)H_1(x) + ((a')^{p^{\ell}}-a+a)x^{n}.
\]
Consequently, we may write the following:
\begin{equation}\label{eq:5.1}
F(x) = h(x)^{p^{\ell}} + \pi\,h(x)H_{1}(x) + ((a')^{p^{\ell}}-a)x^{n}.
\end{equation}
Next, we write $h(x)=g_{1}(x)^{e_{1}}\cdots g_{d}(x)^{e_{d}} + \pi\,H(x), \, e_i>0,$ where $g_{1}(x),\ldots,g_{d}(x)$ are distinct monic polynomials that are irreducible modulo $M_v$ and $H(x)\in R_v[x]$. Substituting this expression into \eqref{eq:5.1}, we obtain
\begin{align*}
F(x)
&= \left(\prod_{i=1}^{d} g_{i}(x)^{e_{i}} + \pi\,H(x)\right)^{p^{\ell}}
   + \pi\,\left(\prod_{i=1}^{d} g_{i}(x)^{e_{i}} + \pi\,H(x)\right)H_{1}(x) + ((a')^{p^{\ell}}-a)x^{n}, \\
&= \left(\prod_{i=1}^{d} g_{i}(x)^{e_{i}}\right)^{p^{\ell}}
   + \pi\,\left(\prod_{i=1}^{d} g_{i}(x)^{e_{i}} + \pi\,H(x)\right)H_{1}(x) + \pi^{2}H_{2}(x) + ((a')^{p^{\ell}}-a)x^{n}
\end{align*}
for some polynomial $H_{2}(x)\in R_v[x]$. Thus, we may write 
$F(x)=\left(g_{1}(x)^{e_{1}}\cdots g_{d}(x)^{e_{d}}+\pi H(x)\right)^{p^{\ell}} + \pi M(x),$ for some $M(x)\in R_v[x]$. Since $\ell>0$, Theorem~\ref{valuation dede} implies that $R_v[\theta]$ is integrally closed if and only if $\overline{M}(x)$ is coprime to $\overline{h}(x)$ modulo $M_v$.
This holds if and only if the polynomial $\overline{a}_1 x^{n}$ is coprime to $\overline{h}(x)$ modulo $M_v$, which is equivalent to the condition $v(a_1) = 0$, where $a_1=\frac{a - (a')^{p^{\ell}}}{\pi}.$ This condition also simplifies to $v(a-(a')^p) = v(\pi)$. This completes the proof in this case.\\[2mm] 
\textbf{Case (vii).} Now consider the last case where $ackm \notin M_v$. As $v(\D_F) > 0$ and $ackm \notin M_v$, we have $v(\mathcal{C} - \mathcal{E}) > 0$. Furthermore, using $v(\mathcal{C}- \mathcal{E}) > 0$ along with $v(km) = 0$, it follows that $v(n(km-n)) = 0$. We claim that there exists an element $\alpha \in R_v$ such that the polynomial $x^t - \overline{\alpha}$ is the product of all distinct monic repeated irreducible factors of $\overline{F}(x)$ over $R_v/M_v$. Let $\alpha_1$ and $\alpha_2$ be as defined in Proposition \ref{2.06} and $\alpha \equiv \alpha_1^{u_1} \alpha_2^{u_2} \pmod{M_v^2}$. Suppose $\beta$ is any repeated root of $\overline{F}(x)= (x^k+\overline{c})^m -\overline{a} x^n$ in the algebraic closure of $R_v/M_v.$ Then using the derivation similar to that for \eqref{0.009}, we have $\beta^k \equiv \alpha_1 \pmod{M_v}$. Clearly, $\beta^k \in R_v/M_v$ and $F(\beta)\equiv 0 \pmod{M_v}$ and using \eqref{special} imply that
    \begin{align*}
        \beta^n \equiv a^{-1} (\beta^k +c)^m \equiv \alpha_{2} \in R_v/M_v.
    \end{align*}
    As $\gcd(n,k)=t$, there exist $u_1,u_2\in\mathbb{Z}$ such that $ku_1+nu_2=t.$ Using this we obtain $\beta^t = (\beta^k)^{u_1}(\beta^n)^{u_2} \in R_v/M_v$, i.e., $\beta^t \equiv \alpha_1^{u_1} \alpha_{2}^{u_2}\pmod{M_v}$. Thus, we have proved that any repeated root of $\overline{F}(x)$ is a root of $x^t-\overline{\alpha}$. Now suppose $\beta_1$ is root $x^t-\alpha$ modulo $M_v$, then $\beta_1$ satisfies $\beta_1^t \equiv \alpha_1^{u_1} \alpha_{2}^{u_2} \pmod{M_v}$. Proposition \ref{2.06} yields $\beta_1$ is a repeated root of $\overline{F}(x)$ if and only if $\mathcal{C} - \mathcal{E} \in M_v$. Hence, using the fact that $v(\mathcal{C} - \mathcal{E}) >  0$, we conclude that every root of $x^t-\overline{\alpha}$ is a repeated root of $F(x)$ modulo $M_v$. Since $t=\gcd(n,k)$, it is easy to see that
    \begin{equation}\label{3.20}
        F(x)= ((x^t)^{k_1}+c)^m -a(x^t)^{n_1} = (x^t-\alpha)q(x) + (\alpha^{k_1}+c)^m - a \alpha^{n_1}
    \end{equation}
    for some $q(x) \in R_v[x^t]$. As $x^t - \overline{\alpha}$ divides $\overline{F}(x)$, we have $\overline{F}(x) = (x^t-\overline{\alpha})\overline{q}(x)$. Let $\overline{F}(x)=\overline{g}_1(x)^{e_1}\cdots \overline{g}_t(x)^{e_t}$ be the factorization of $\overline{F}(x)$ into a product of powers of distinct irreducible polynomials over $R_v/M_v$ with each $g_{i}(x)\in R_v[x]$ monic. If necessary, after renaming, assume that $e_{i}>1$ for $1\leq i\leq t_{1}$ and $e_i=1$ for $t_1< i \leq t$. Then $x^t-\overline{\alpha}=\prod\limits_{i=1}^{t_1}\overline{g}_i(x)$. Write\\[2mm]
    $x^t-\alpha=\prod\limits_{i=1}^{t_1} g_i(x)+\pi H_1(x), \quad q(x)=\prod\limits_{i=1}^{t_1}g_i(x)^{e_i-1} \prod\limits_{i=t_1+1}^{t}g_i(x)+\pi H_2(x)$ \\[2mm] 
    for some $H_1(x),H_2(x)\in R_v[x]$. Substituting from the above equation into $\eqref{3.20}$, we have
    \begin{align*}
        F(x)=&\prod_{i=1}^{t}g_i(x)^{e_i}
+ \pi H_1(x)\prod_{i=1}^{t_1}g_i(x)^{e_i-1} \prod_{i=t_1+1}^{t}g_i(x)+ \pi H_2(x)\prod_{i=1}^{t_1}g_i(x) \nonumber\\ &+ \pi^2H_1(x)H_2(x) + (\alpha^{k_1}+c)^m - a \alpha^{n_1}.
    \end{align*}
 Clearly, each summand on the right-hand side of the above equation except possibly $(\alpha^{k_1}+c)^m - a \alpha^{n_1}$ belongs to $\langle \pi, g_i(x) \rangle^2$ for $1\leq i\leq t_1$. So $F(x) \in \langle \pi, g_i(x) \rangle^2$ for some $i$ ($1\leq i\leq t_1$) if and only if $(\alpha^{k_1}+c)^m - a \alpha^{n_1} \in M_v^2$. As $\overline{F}(x)$ is not divisible by $\overline{g}_i(x)^2$ for $t_1< i\leq t$, it is clear that $F(x)\not\in  \langle \pi, g_i(x) \rangle^2$ for such $i$. So $F(x)\not \in  \langle \pi, g_i(x) \rangle^2$ for any $i$ ($1 \leq i \leq t$) if and only if $(\alpha^{k_1}+c)^m - a \alpha^{n_1} \notin M_v^2$. To complete this case, it only remains to prove that $(\alpha^{k_1}+c)^m - a \alpha^{n_1} \in M_v^2$ if and only if $\mathcal{C} - \mathcal{E} \in M_v^2$. This follows from Proposition \ref{2.06} and hence the theorem is proved.
\end{proof}

\noindent\textbf{Acknowledgements.}
The author sincerely thank Professor Shanta Laishram and Dr. Prabhakar Yadav for carefully reading an earlier version of this paper. Their valuable suggestions and helpful discussions through email greatly improved the presentation of the paper. 

\bibliographystyle{amsplain}
\bibliography{references}  
\end{document}